\documentclass[10pt,letterpaper]{amsart}

\usepackage{geometry}
\usepackage[dvipsnames]{xcolor}

\usepackage[inline,shortlabels]{enumitem}
\setlist{noitemsep}

\usepackage{amsmath,amssymb,amsthm}
\usepackage{mathtools}
\usepackage{stmaryrd}
\usepackage{textgreek}
\usepackage[retainorgcmds]{IEEEtrantools}
\usepackage{colonequals}
\usepackage{relsize}
\usepackage[bbgreekl]{mathbbol}
\usepackage{mathrsfs}
\usepackage{tikz-cd}
\usepackage{array}

\usepackage{todonotes}

\DeclareSymbolFontAlphabet{\mathbb}{AMSb} %to ensure that the meaning of \mathbb does not change
\DeclareSymbolFontAlphabet{\mathbbl}{bbold}
\usepackage{graphicx}

\makeatletter
\newcommand{\smallbullet}{} % for safety
\DeclareRobustCommand\smallbullet{%
  \mathord{\mathpalette\smallbullet@{0.7}}%
}
\newcommand{\smallbullet@}[2]{%
  \vcenter{\hbox{\scalebox{#2}{$\m@th#1\bullet$}}}%
}
\makeatother
\usepackage[pdfusetitle,colorlinks,unicode]{hyperref}
\hypersetup{citecolor=black,linkcolor=black,urlcolor=black}
\usepackage[capitalise,noabbrev]{cleveref}
\crefformat{enumi}{#2\textup{(#1)}#3}
\crefformat{equation}{\ensuremath{(#2#1#3)}}
\crefmultiformat{equation}{\ensuremath{(#2#1#3)}}{ and~\ensuremath{(#2#1#3)}}{, \ensuremath{(#2#1#3)}}{, and~\ensuremath{(#2#1#3)}}
\numberwithin{figure}{section}
\numberwithin{equation}{section}
\newtheorem{theorem}[figure]{Theorem}
\AddToHook{env/theorem/begin}{\crefalias{figure}{theorem}}
\newtheorem{lemma}[figure]{Lemma}
\AddToHook{env/lemma/begin}{\crefalias{figure}{lemma}}
\newtheorem{corollary}[figure]{Corollary}
\AddToHook{env/corollary/begin}{\crefalias{figure}{corollary}}
\newtheorem{proposition}[figure]{Proposition}
\AddToHook{env/proposition/begin}{\crefalias{figure}{proposition}}

\AddToHook{env/conjecture/begin}{\crefalias{figure}{conjecture}}

\AddToHook{env/question/begin}{\crefalias{figure}{question}}

\theoremstyle{definition}
\newtheorem{definition}[figure]{Definition}
\AddToHook{env/definition/begin}{\crefalias{figure}{definition}}
\newtheorem{notation}[figure]{Notation}
\AddToHook{env/notation/begin}{\crefalias{figure}{notation}}
\newtheorem{remark}[figure]{Remark}
\AddToHook{env/remark/begin}{\crefalias{figure}{remark}}
\newtheorem{example}[figure]{Example}
\AddToHook{env/example/begin}{\crefalias{figure}{example}}
\newtheorem{construction}[figure]{Construction}
\AddToHook{env/construction/begin}{\crefalias{figure}{construction}}

\DeclareMathOperator{\gr}{gr}

\DeclareSymbolFontAlphabet{\mathbb}{AMSb} %to ensure that the meaning of \mathbb does not change
\DeclareSymbolFontAlphabet{\mathbbl}{bbold}

\DeclareMathOperator{\Spec}{Spec}

\DeclareMathOperator{\conj}{conj}

\DeclareMathOperator{\Fil}{Fil}
\DeclareMathOperator{\dR}{dR}

\newcommand{\op}{\mathrm{op}}
\newcommand{\perf}{\mathrm{perf}}

\newcommand{\id}{\mathrm{id}}

\newcolumntype{C}[1]{>{\centering\arraybackslash}p{#1}}

\mathchardef\mhyphen="2D

\newcommand\restr[2]{{\left.\kern-\nulldelimiterspace#1\vphantom{\big|}\right|_{#2}}}
\newcommand{\suchthat}{\;\ifnum\currentgrouptype=16 \middle\fi\vert\;}
\newcommand{\cosimp}[3]{\xymatrix@1{#1 \ar@<.4ex>[r] \ar@<-.4ex>[r] & {\ }#2 \ar@<0.8ex>[r] \ar[r] \ar@<-.8ex>[r] & {\ } #3  \cdots }} 
\usepackage[all]{xy}

\AtBeginDocument{%
   \def\MR#1{}
}

\title{Automorphisms of Frobenius twisted de Rham cohomology}

\author{Shizhang Li and Shubhodip Mondal}

\address[Shizhang Li]{Morningside Center of Mathematics and State Key Laboratory of Mathematics Sciences,
Academy of Mathematics and Systems Science, Chinese Academy of Sciences, Beijing 100190, China}
\email{lishizhang@amss.ac.cn}

\address[Shubhodip Mondal]{University of British Columbia, 1984 Mathematics Rd, Vancouver, BC V6T 1Z2, Canada}
\email{smondal@math.ubc.ca}

\date{First version August 1, 2025}

\begin{document}

\begin{abstract}
In this short paper, we prove that the moduli of automorphisms of Frobenius twisted de Rham cohomology functor
is given by $\mathbb{G}_m$. 
Our method is to use the notion of $\mathbb{G}_a^{\mathrm{perf}}$-modules and its connection to the de Rham cohomology functor introduced in \cite{M22}.
%\textcolor{red}{delete}which we will revisit. 
As an application of the induced $\mathbb{G}_m$-action, 
we reprove a result of Bhatt, Petrov, and Vologodsky on the decomposition of the Frobenius twisted de Rham complex. 
\end{abstract}

\maketitle
\tableofcontents

\section{Introduction}
In \cite{LM24}, we described the moduli of endomorphisms of de Rham cohomology in a variety of scenarios and deduced Drinfeld's refinement \cite[\S~5.12.1]{drinfeld} to the Deligne--Illusie decomposition theorem \cite{DI87} as an application (the latter conclusion is also due to Bhatt--Lurie, which uses ``prismatization"; see \cite[Rmk.~4.7.18]{BL22a} and \cite[Rmk.~5.16]{BL22b}). 
In this paper, inspired by the results in \cite[Rmk~2.7.4, 2.7.5]{fg} and \cite{Petrov25}, we consider the following moduli problem (\cref{moduliproblem}). 

\begin{definition}\label{def1}
Let $k$ be a perfect field. 
We use $F^*F_* \mathrm{dR}$ to denote
the functor $(\mathrm{Sch}_k^{\mathrm{sm}})^{\mathrm{op}} \to 
\mathrm{CAlg}(k)$, defined by the following assignment $$(\mathrm{Sch}_k^{\mathrm{sm}})\ni X \mapsto F_{X/k}^* F_{X/k,*} \Omega^*_{X/k},$$ where $\mathrm{CAlg}(k)$ denotes the $\infty$-category of $E_\infty$-algebras over $k$.
\end{definition}{}

\begin{definition}[Moduli of automorphisms]\label{moduliproblem} Let $k$ be a perfect field. 
Let $\mathrm{Alg}_{k}$ denote the category of ordinary commutative $k$-algebras. We define a moduli functor $$\mathcal{M}^{\mathrm{Aut}}_{}\colon \mathrm{Alg}_{k} \to \infty\mathrm{-Groupoids}$$ by the assignment 
$$\mathrm{Alg}_{k} \ni B \mapsto \mathrm{Aut}(F^*F_*\mathrm{dR} \otimes_k B),$$ where $F^*F_*\mathrm{dR} \otimes_k B \colon (\mathrm{Sch}_k^{\mathrm{sm}})^{\mathrm{op}} \to 
\mathrm{CAlg}(B) $ is the functor base changed from \cref{def1}. 
\end{definition}

Our main new theorem is the following.

\begin{theorem}\label{mainthm1}
 In the above set up, we have an isomorphism  $$\mathcal{M}^{\mathrm{Aut}} \simeq \mathbb{G}_m,$$ compatible with the group structures.   
\end{theorem}

To prove our result, we use quasi-syntomic descent, and general techniques for understanding certain ``cohomological" functors on quasi-regular semiperfect algebras (which includes e.g., de Rham cohomology) as in \cite{M22}, which is renamed ``semiperfect transmutation" in this paper. The key idea in the proof of \cref{mainthm1} is to use this transmutation construction to reduce the calculation of automorphisms of $F^* F_* \dR$ to the calculation of automorphisms of a simpler, more tractable object, called an \emph{augmented $\mathbb{G}_a^\perf$-module} (see \cref{sec3}).
The results of this paper, combined with that of \cite{LM24} are presented in the table below.

\begin{center}

\bgroup
\def\arraystretch{1.5}
\begin{tabular}{ | C{17em} | C{17em}| } 
\hline
\textbf{Cohomology theory} & \textbf{Moduli of automorphisms} \\ 

\hline
${\mathrm{dR}\colon (\mathrm{Sch}_k^{\mathrm{sm}})^{\mathrm{op}} \to 
\mathrm{CAlg}(k)}$ & $\left \{\mathrm{id} \right \} $ \\ 
\hline
${\mathrm{dR}^n\colon (\mathrm{Sch}_{W_n(k)}^{\mathrm{sm}})^{\mathrm{op}} \to 
\mathrm{CAlg}(k)}, n \ge 2 $ & ${\mathbb{G}_m^\sharp}$ \\
\hline
${F^*F_*\mathrm{dR}\colon (\mathrm{Sch}_k^{\mathrm{sm}})^{\mathrm{op}} \to 
\mathrm{CAlg}(k)}$ & $\mathbb{G}_m$\\

\hline

\hline
\end{tabular}
\egroup
\end{center}
\vspace{2mm}
As a consequence of \cref{mainthm1}, we obtain a $\mathbb{G}_m$-action on the functor $F^* F_* \mathrm{dR}$. The following result describes how this $\mathbb{G}_m$-action interacts with the conjugate filtration induced on $F^*F_* \dR$ (see \Cref{conjugate filtration}). 

\begin{theorem}
\label{mainthm2}
The $\mathbb{G}_m$-action preserves the conjugate filtration on $F^*F_* \dR$ ,
and $\gr^{\mathrm{conj}}_i F^* F_* \dR [i]$ is a $\mathbb{G}_m$-representation purely of weight $i$.
\end{theorem}

As a consequence of \cref{mainthm2}, we obtain \begin{corollary}[c.f.~{\cite[Rmk.~2.7.5]{fg} and \cite[Thm.~1.2]{Petrov25}}]\label{maincoro}
For a smooth $k$-scheme $X$, there is a functorial decomposition in $D(X, \mathcal{O}_{X^{(1)}})$ 
\[
F^* F_* \mathrm{dR}_X \cong \bigoplus_{i \geq 0} F^* \Omega^i_{X^{(1)}/k}[-i].
\]
\end{corollary}
\begin{remark} As observed in \cite{Petrov25}, \cref{maincoro} implies that the de Rham complex decomposes for $F$-split varieties. We point out that in \cite[Thm.~1.2]{Petrov25}, Petrov proves a decomposition statement after pullback to $F_* W(\mathcal{O}_X)/p$, which proves the decomposition statement for quasi-$F$-split varieties as in \cite[Thm.~1.1]{Petrov25}. The main difference in our approach is to deduce \cref{maincoro} from the moduli description in \cref{mainthm1}, which is proven by using the notion of augmented $\mathbb{G}_a^\perf$-modules.
\end{remark}

\section{Preliminary reductions}

Let us fix $B \in \mathrm{Alg}_k$ and consider the functor $F^* F_* \mathrm{dR} \otimes_k B.$ 
By left Kan extension, we can extend it to a functor still denoted as $$F^* F_* \mathrm{dR} \otimes_k B \colon
(\mathrm{Sch}^{\mathrm{qsyn}}_k)^{\mathrm{op}} \to \mathrm{CAlg}(D(B)).$$ 
\begin{proposition}[Quasisyntomic descent]
\label{qsyndescent}
The above functor $F^* F_* \mathrm{dR} \otimes_k B$ \emph{satisfies quasisyntomic descent}.
\end{proposition}

\begin{proof}
By using the conjugate filtration on derived de Rham cohomology, and the fact that for a quasi-syntomic $k$-algebra $S$, the contangent complex $\mathbb{L}_{S/k}$ as an $S$-module has Tor amplitude in cohomological degrees $[-1,0]$, we deduce that the functor $F^* F_* \mathrm{dR} \otimes_k B $ is coconnective. Since totalizations of coconnective objects commute with filtered colimits, to prove descent for $F^* F_* \mathrm{dR} \otimes_k B$, it suffices to prove descent for $F^* F_* \mathrm{dR}$ (as $k$ is a field, $B$ is a filtered colimit of finite dimensional $k$-vector spaces). So without loss of generality, we shall assume that $B\coloneqq k$ throughout the proof.

To this end, since conjugate filtration is increasing and exhaustive, by commuting totalizations of connective objects with filtered colimits, it would be enough to prove that the functor 
\begin{equation}\label{descentforcotangent}
 (\mathrm{Sch}_k^{\mathrm{Aff,qsyn}})^{\mathrm{op}} \ni S \mapsto \wedge^i\mathbb{L}_{S/k} \otimes_{S, \varphi_S}S \in D(k)   
\end{equation}satisfies quasi-syntomic descent. We will argue for $i=1$ and explain how to deduce the $i>1$ case from it. Let $R \to S$ be a quasisyntomic cover. We will prove descent for this cover, i.e., we will show that there is a natural isomorphism
$$\mathbb{L}_{R/k} \otimes_{R, \varphi_R}  R \simeq \mathrm{Tot} (\mathbb{L}_{S^\bullet/k} \otimes_{S^\bullet, \varphi}S^\bullet),$$ where $S^\bullet$ denotes the cosimplicial ring obtained by taking the Cech nerve of $R \to S$. Note that we have a fiber sequence in the category of cosimplicial objects of $D(k)$ of the form
\begin{equation}\label{imptriangle}
 \mathbb{L}_{R/k} \otimes_R  S^\bullet\otimes_{S^\bullet, \varphi}S^\bullet \to \mathbb{L}_{S^\bullet/k} \otimes_{S^\bullet, \varphi}S^\bullet \to  \mathbb{L}_{S^\bullet/R} \otimes_{S^\bullet, \varphi}S^\bullet.
\end{equation}
By faithfully flat descent, 
$$\mathrm{Tot} (\mathbb{L}_{R/k} \otimes_R  S^\bullet\otimes_{S^\bullet, \varphi}S^\bullet) \simeq \mathrm{Tot} (\mathbb{L}_{R/k} \otimes_{R, \varphi_R}  R\otimes_{R}S^\bullet )\simeq \mathbb{L}_{R/k} \otimes_{R, \varphi_R}  R.$$Therefore, in order to prove descent, we need to show that 
\begin{equation}\label{eq879}
 \mathrm{Tot}(\mathbb{L}_{S^\bullet/R} \otimes_{S^\bullet, \varphi}S^\bullet) \simeq 0  . 
\end{equation}{}
Note that for each $[n] \in \Delta,$ $S^{[n]}$ appearing in $S^\bullet$ is quasisyntomic, and thus $\mathbb{L}_{S^{[n]}/k}$ has Tor amplitiude in cohomological degrees $[-1,0]$ as $S^{[n]}$-module. By using the transitivity triangle for the cotangent complex associated to the maps $k \to R \to S^{[n]}$, it follows that each term $\mathbb{L}_{S^{[n]}/R} \otimes_{S^{[n]}, \varphi} S^{[n]}$ is concentrated in cohomological degrees $[-2,0]$. To prove \cref{eq879}, it suffices to show that
$\mathrm{Tot}(\mathbb{L}_{S^\bullet/R} \otimes_{S^\bullet, \varphi}S^\bullet) \otimes_{R}S \simeq 0.$ Let $S \to T^\bullet$ denote the base change of $R \to S^\bullet$ along $R \to S.$ Since $R \to S$ is flat, one can write $S$ as a filtered colimit of finite free $R$-modules. By commutating totalization of (cohomologically) uniformly bounded below objects with filtered colimits and using base change properties of the cotangent complex, we have 
$$\mathrm{Tot}(\mathbb{L}_{S^\bullet/R} \otimes_{S^\bullet, \varphi}S^\bullet) \otimes_{R}S \simeq \mathrm{Tot} (\mathbb{L}_{T^\bullet/S}\otimes_{T^\bullet, \varphi} T^\bullet).$$
Note that $S \to T^\bullet$ can be viewed as a split augmented cosimplicial $S$-algebra since it is obtained from Cech nerve of $S \to S \otimes_R S$ which admits a section. 
By \cite[Example 3.11]{Mat16},
this implies that $$\mathrm{Tot} (\mathbb{L}_{T^\bullet/S}\otimes_{T^\bullet, \varphi} T^\bullet) \simeq \mathbb{L}_{S/S}\otimes_{S, \varphi_S} S \simeq 0,$$ as desired. This finishes the proof in the case $i=1$. For $i >1$, one argues similarly by using the graded pieces of the finite filtration on $\wedge^i \mathbb{L}_{S^\bullet/k} \otimes_{S^\bullet, \varphi}S^\bullet$ induced from \cref{imptriangle} (see \cite[p.~213]{BMS19}).
\end{proof}

\begin{remark}
    The argument for proving the descent statement for \cref{descentforcotangent} essentially follows the same idea in \cite[Thm.~3.1]{BMS19}, with the technical modification of using split cosimplicial objects to avoid talking about ``homotopy of cosimplicial objects" in the $\infty$-categorical set up as in our situation. A flat descent statement for the functor in \cref{descentforcotangent} appears in \cite[Lem.~5.10]{Petrov25}.
\end{remark}{}

%We will need to prove that $$\tau^{\le r} \mathrm{Tot}(\mathbb{L}_{S^\bullet/R} \otimes_{S^\bullet, \varphi}S^\bullet) \simeq \mathrm{Tot}(\tau^{\le r} (\mathbb{L}_{S^\bullet/R} \otimes_{S^\bullet, \varphi}S^\bullet)) \simeq 0$$ for all $ r \in \mathbb{Z}$. When $r < -2$, the desired vanishing directly holds.

Let $\mathrm{QRSP}_k$ denote the cateory of semi-perfect $k$-algebras which are quasi-regular over $k$.\footnote{From now on we will call them \emph{qrsp over $k$}.} They form a basis for the quasisyntomic topology on $\mathrm{Sch}^{\mathrm{qsyn}}_k$. 
For more discussions on this notion, we refer the readers to \cite[\S 4]{BMS19}.

\begin{proposition}
\label{discrete value on QRSP}
If $S \in \mathrm{QRSP}_k$, then $F^* F_* \mathrm{dR}(S) \otimes_k B$ is a discrete ring.  
\end{proposition}

\begin{proof}
This follows from the conjugate filtration on (Frobenius twisted) de Rham cohomology (see \cite[Proposition 3.5]{Bha12}) 
and the fact that $\mathbb{L}_{S/k}$ is isomorphic to a flat $S$-module concentrated in cohomological degree $-1$.  
\end{proof}

\begin{corollary}
The space of automorphisms $\mathrm{Aut}(F^* F_* \mathrm{dR} \otimes_k B)$ is discrete,
and similarly for the space of endomorphisms $\mathrm{End}(F^* F_* \mathrm{dR} \otimes_k B)$.
\end{corollary}

\begin{proof}
This follows from
\Cref{qsyndescent}, \Cref{discrete value on QRSP} by considering right extensions and using as in the proof of \cite[Lemma 3.3]{LM24}.
\end{proof}

In the proof of \Cref{discrete value on QRSP},
we have used the conjugate filtration on $p$-adic derived de Rham cohomology due to
Bhatt. By pulling back along Frobenius, we have an induced filtration on Frobenius twisted de Rham cohomology.

\begin{notation}
\label{conjugate filtration}
Let us denote the Frobenius pullback of the conjugate filtration as
\[
\Fil^{\conj}_{\bullet}(F^* F_* \mathrm{dR} \otimes_k B) \coloneqq F^*(\Fil^{\conj}_{\bullet}F_* \mathrm{dR}) \otimes_k B.
\]
If there is no risk of confusion, we still refer to the above as the conjugate filtration
on $F^* F_* \mathrm{dR} \otimes_k B$.
\end{notation}

Therefore we may view $F^* F_* \mathrm{dR} \otimes_k B$ as a filtered $\Fil^{\conj}_0 = \mathcal{O} \otimes_k B$-algebra.

\begin{proposition}
\label{automorphism is linear over Fil0}
Any endomorphism $\sigma \in \mathrm{End}(F^*F_*\mathrm{dR} \otimes_k B)$ must preserve the conjugate filtration.
Any automorphism of $\mathcal{O} \otimes_k B$ is identity,
therefore any automorphism of $F^*F_*\mathrm{dR} \otimes_k B$ must be linear over $\mathcal{O} \otimes_k B$.
\end{proposition}

\begin{proof}
The first sentence follows from the same proof of \cite[Lemma 5.3]{LM24}.
For the second sentence: any automorphism of $\mathcal{O} \otimes_k B$ induces an automorphism of $B[x]$
by applying to the algebra $k[x]$. By functoriality, this would necessarily be determined by an automorphism of $\mathbb{A}^1_B$
as a ring scheme, hence must be identity. Since for any $k$-algebra $R$, the algebra $R \otimes_k B$ is generated by
subalgebras of the form $f(k[x]) \otimes_k B$ for varying maps $f \colon k[x] \to R$,
we conclude again by functoriality that the induced automorphism on $R \otimes_k B$ is identity.
\end{proof}

\section{Augmented $\mathbb{G}_a^{\mathrm{perf}}$-modules and transmutation}\label{sec3}

Below, we fix a $k$-algebra $B$ and introduce some definitions that will be required.

\begin{definition}
\label{Gaperf definition}
We denote $\mathbb{G}_a^{\mathrm{perf}} \coloneqq \Spec(\mathbb{F}_p[x^{1/p^\infty}])$
the ring scheme representing the functor sending any $\mathbb{F}_p$-algebra $A$
to its inverse limit perfection $A^{\flat} \coloneqq \lim_{a \mapsto a^p} A$.
By base change to $\Spec(B)$, we again obtain a ring scheme denoted by $\mathbb{G}_{a,B}^{\mathrm{perf}}$.
When there is no risk of confusion, we will omit the subscript $B$
and still denote the latter by $\mathbb{G}_a^{\mathrm{perf}}$.
Its underlying scheme structure is given by $\mathrm{Spec}(B[x^{1/p^\infty}])$.  
\end{definition}{}

\begin{definition}[{\cite[Definition 2.2.5]{M22}}]
A $\mathbb{G}_a^{\mathrm{perf}}$-module is a $\mathbb{G}_a^{\mathrm{perf}}$-module 
object in the category $\mathrm{Aff}_B$ of affine schemes over $\Spec(B)$.
\end{definition}

\begin{remark}[{\cite[Remark 2.2.8]{M22}}]
\label{interprete Gaperf module}
Let us view $\mathbb{G}_a^{\mathrm{perf}}$ as a monoid scheme. Then the category of
affine schemes over $\Spec(B)$ together with an action of this monoid is anti-equivalent to
$\mathbb{N}[1/p]$-graded $B$-algebras.
The category of abelian group objects in the former category is then anti-equivalent to
$\mathbb{N}[1/p]$-graded bi-commutative Hopf-$B$-algebras.
Finally the category of $\mathbb{G}_a^{\mathrm{perf}}$-modules is anti-equivalent to $\mathbb{N}[1/p]$-graded
bi-commutative Hopf $B$-algebras
which satisfy a compatibility (analogous to the diagram in \cite[Remark 2.1.8]{M22})
between summation on $\mathbb{G}_a^{\mathrm{perf}}$ and the graded Hopf-$B$-algebra structure.
As a consequence of this compatibility, one can show that the degree $0$ component, which is a priori an arbitrary
Hopf-$B$-algebra, is actually given by $B$ (see \cite[Proposition 2.2.9]{M22}).
\end{remark}

\begin{definition}[Augmented $\mathbb{G}_a^{\mathrm{perf}}$-module]
\label{augmodule}
An augmented $\mathbb{G}_a^{\mathrm{perf}}$-module is a $\mathbb{G}_a^{\mathrm{perf}}$-module $M$
together with a map $d\colon M \to \mathbb{G}_a^{\mathrm{perf}}$ of $\mathbb{G}_a^{\mathrm{perf}}$-modules. 
\end{definition}

\begin{remark}
 The above definition appeared in {\cite[Definition 2.2.14]{M22}} where it was called pointed $\mathbb{G}_a^\perf$-module.
\end{remark}{}

 We recall the following construction, which is a generalization of \cite[\S~3.4]{M22} to an arbirary base ring $B$ of characteristic $p$. This construction was called ``unwinding" in \cite{M22}, which is called ``semiperfect transmutation" or simply ``transmutation" here, following the more recent terminology in \cite{fg}. Morally speaking, these constructions allow one to build a cohomology theory out of a \emph{single} (!) highly structured object, namely an augmented $\mathbb{G}_a^\perf$-module (as in \cite{M22}), or a ring stack (as in \cite{fg}, \cite{LM24}).

\begin{construction}[Semiperfect transmutation]
\label{sections}
  Let $\mathscr{P}I$ denote the category of pairs $(S,I)$ where $S$ is a perfect $k$-algebra and $I$ is an ideal of $S$. 
  Let $d \colon M \to \mathbb{G}_a^{\mathrm{perf}}$ be an augmented $\mathbb{G}_a^{\mathrm{perf}}$-module over $B$.
  Below, out of the data $(M, d)$, we shall construct a functor denoted as
  \[
  \mathrm{Tm}(d) \colon \mathscr{P}I \to \mathrm{Alg}_B.
  \]
  Let $(S, I) \in \mathscr{P}I$, we first construct an affine scheme $\mathrm{Sec}_{(S,I)}(d)$ over 
  $\mathrm{Spec}(S\otimes_k B)$ that roughly speaking, is the moduli of \emph{sections of $d\colon M \to \mathbb{G}_a^{\mathrm{perf}}$ over the ideal} $I$. 
  More precisely, note that we have a functor 
  \[
  \mathrm{Aff}^{\mathrm{op}}_{S \otimes_k B} \to \mathrm{Sets}
  \]
  defined by 
  \[
  \mathrm{Aff}^{\mathrm{op}}_{S \otimes_k B}\ni \mathrm{Spec}(A) \mapsto  \mathrm{Eq} \left(\mathrm{Hom}_S(I, M(A)) \,\, \substack{\xrightarrow{d} \\ \xrightarrow{*}}\,\,  \mathrm{Hom}_S(I, A^\flat)\right).
  \]
  Recall that $A^\flat \coloneqq \mathbb{G}_a^{\mathrm{perf}}(A)$. Let us explain the above definition:
  Applying the limit perfection functor to the natural maps
  $S \to S\otimes_k B \to A$, we get a composite map $S \to A^\flat$.
  Therefore we get a natural $S$-module structure on the $A^\flat$-module $M(A)$. The map $*$ denotes the constant map to the element $* \in \mathrm{Hom}_S(I, A^\flat)$ given by the composition $I \subseteq S \to A^\flat$.
  In other words, we are looking at the set of dashed arrows that makes the diagram 
  \begin{center}
      \begin{tikzcd}
                    & I \arrow[ld, dotted, bend right] \arrow[d, "*"] \\
M(A) \arrow[r, "d"] & A^\flat                                        
\end{tikzcd}
  \end{center}
  commutative. By the adjoint functor theorem, the functor $\mathrm{Aff}^{\mathrm{op}}_{S \otimes_k B} \to \mathrm{Sets}$ is representable by an affine scheme over $\mathrm{Spec}\, (S \otimes_k B)$ which we define to be $\mathrm{Sec}_{(S,I)}(d)$. The functor $\mathrm{Tm}(d)$ is defined by 
$$ \mathscr{P}I \ni (S, I) \mapsto \Gamma (\mathrm{Sec}_{(S,I)}(d), \mathcal{O}).$$ 
By definition, this construction is covariantly functorial in the pair $(S, I)$
and  contravariantly functorial in the augmented $\mathbb{G}_a^{\mathrm{perf}}$-module $(M, d)$.
Hence the above construction produces a functor
\[
\mathrm{Tm}\colon \Big(\mathbb{G}_a^{\mathrm{perf}}-\mathrm{Mod}_*\Big)^{\op} \to \mathrm{Fun}(\mathscr{P}I, \mathrm{Alg}_B),
\] by the assignment $d \mapsto \mathrm{Tm}(d),$
where the source denotes the opposite category of the category of augmented $\mathbb{G}_a^\mathrm{perf}$-modules. 

\end{construction}

\begin{example}\label{useexample}
\label{transmutation on perfect}
The category $\mathscr{P}I$ contains the category $\mathrm{Perf}_k$ of perfect $k$-algebras
via $S \mapsto (S, (0))$.
The functor of points of $\mathrm{Sec}_{(S,(0))}(d)$ is given by the singleton, therefore
we have $\mathrm{Tm}(d)(S, (0)) \simeq S \otimes_k B$ for any augmented $\mathbb{G}_a^\perf$-module $(M, d)$.
\end{example}

\begin{example}
\label{transmutation on generator}
The category $\mathscr{P}I$ has a particular object, which will be of interest to us later, given by 
$(S,I) \coloneqq (k[x^{1/p^\infty}], x)$. We will calculate the value of $\mathrm{Tm}(d)$ evaluated at this object.
Since the ideal $I$ is free rank $1$ as an $S$-module, the scheme $\mathrm{Sec}_{(S,I)}(d)$
represents the functor $\mathrm{Aff}^{\mathrm{op}}_{S \otimes_k B}\ni \mathrm{Spec}(A) \mapsto \mathrm{Eq} \left(M(A) \,\, \substack{\xrightarrow{d} \\ \xrightarrow{*}}\,\,  A^\flat\right)$.
Unwinding definitions, this amounts to finding dashed arrows of algebras making the diagram below commutative:
\[
\xymatrix{
B[x^{1/p^\infty}] \ar[r]^-{\simeq} \ar[d]^-{d^*} & S \otimes_k B \ar[d] \\
\mathcal{O}(M) \ar@{-->}[r] & A.
}
\]
Therefore we see that $\mathrm{Sec}_{(S,I)}(d) \simeq M$ with structure map to $\Spec(S \otimes_k B)$
given by $d$, and hence we have an identification
\[
\xymatrix{
\mathrm{Tm}(d)(k[x^{1/p^\infty}], (0)) \ar[d]_-{\mathrm{Tm}(d)(x \mapsto x)} \ar[r]^-{\simeq} & 
\mathcal{O}(\mathbb{G}_{a}^{\mathrm{perf}}) \ar[d]^-{d^*} \\
\mathrm{Tm}(d)(k[x^{1/p^\infty}], (x)) \ar[r]^-{\simeq} & \mathcal{O}(M)
}
\]
compatible with the identification in \Cref{transmutation on perfect}.
\end{example}

\begin{remark}
\label{Frob for augmented Gaperf mod}
Using \Cref{transmutation on generator} and the functoriality of \cref{sections} applied to the map of pairs
$F \colon (k[x^{1/p^\infty}], x) \to (k[x^{1/p^\infty}], x)$ which sends
$x \mapsto x^p$, we see that for any $(M, d) \in \mathbb{G}_a^{\mathrm{perf}}-\mathrm{Mod}_*$, there is a natural map $\mathcal{O}(M) \xrightarrow{\text{``Frob''}} \mathcal{O}(M)$ 
which fits into a natural commutative diagram of $B$-algebras:
\[
\xymatrix{
 \mathcal{O}(M) \ar[r]^-{\text{``Frob''}} & \mathcal{O}(M) \\
B[x^{1/p^\infty}] \ar[u] \ar[r]^-{x \mapsto x^p} & B[x^{1/p^\infty}] .\ar[u].
}
\]
\end{remark}

\begin{remark}
\label{why quasi-ideal}
Let $(M, d)$ be an augmented $\mathbb{G}_{a, B}^{\mathrm{perf}}$-module.
Unraveling definitions, we see that the identifications in 
\Cref{transmutation on perfect} and \Cref{transmutation on generator}
make the following diagram commute:
\[
\xymatrix{
\mathcal{O}(M) \ar[r]^-{\simeq} \ar[d]^-{\mathrm{act}^*} &
\mathrm{Tm}(d)(k[x^{1/p^\infty}], (x)) \ar[r]^-{\mathrm{Tm}(d)(x \mapsto x \otimes x)} & 
\mathrm{Tm}(d)(k[x^{1/p^\infty}] \otimes_k k[x^{1/p^\infty}], (1 \otimes x)) \\
\mathcal{O}(\mathbb{G}_{a, B}^{\mathrm{perf}} \times_{\Spec(B)} M) \ar[r]^-{\simeq}
& (k[x^{1/p^\infty}] \otimes_k B) \otimes_B \mathcal{O}(M) \ar[r]^-{\simeq} &
\mathrm{Tm}(d)(k[x^{1/p^\infty}], (0)) \otimes_B \mathrm{Tm}(d)(k[x^{1/p^\infty}], (x)). \ar[u]
}
\]
\end{remark}

\begin{example}
\label{transmutation of id}
The category $(\mathbb{G}_a^{\mathrm{perf}}-\mathrm{Mod}_*)^{\op}$ has initial object given by
$\mathbb{G}_a^{\mathrm{perf}} \xrightarrow{\mathrm{id}} \mathbb{G}_a^{\mathrm{perf}}$.
Since $d \colon M(A) \to A^{\flat}$ in this case is an isomorphism,
the functor of $\mathrm{Sec}_{(S,I)}(d)$ is given by singleton.
Therefore we have $\mathrm{Tm}(\id)(S, I) = S \otimes_k B$. 
\end{example}

\begin{remark}[Base change]
\label{base change remark}
The transmutation process also satisfies two base change properties:
let $S \to T$ be a map of perfect $k$-algebras such that $I \otimes_S T \simeq I \cdot T$,
then for any ideal $I \subset S$ and $d \in (\mathbb{G}_a^{\mathrm{perf}}-\mathrm{Mod}_*)^{\op}$, we have
$\mathrm{Tm}(d)(T, I \cdot T) \simeq \mathrm{Tm}(d)(S, I) \otimes_S T$.
Let $B \to B'$ be a map of $k$-algebras, then for any $d \in (\mathbb{G}_{a, B}^{\mathrm{perf}}-\mathrm{Mod}_*)^{\op}$ by base
change we obtain a $d' \in (\mathbb{G}_{a, B'}^{\mathrm{perf}}-\mathrm{Mod}_*)^{\op}$,
and we have a natural equivalence $\mathrm{Tm}(d') \simeq \mathrm{Tm}(d) \otimes_B B'$ as functors $  \mathscr{P}I \to \mathrm{Alg}_{B'}$.
\end{remark}

\begin{definition}
\label{defoft}
Note that there is a functor $\mathrm{QRSP}_k \to \mathscr{P}I$ defined by 
\[
\mathrm{QRSP}_k \ni S \mapsto (S^\flat, \mathrm{Ker}(S^\flat \twoheadrightarrow S)).
\]
This induces a functor also denoted as
\[
\mathrm{Tm}\colon \Big(\mathbb{G}_a^{\mathrm{perf}}-\mathrm{Mod}_*\Big)^{\op} \to \mathrm{Fun}(\mathrm{QRSP}_k, \mathrm{Alg}_B).
\]
\end{definition}

\begin{example}
Consider the quasi-ideal $\mathbb{G}_a^{\mathrm{perf}} \xrightarrow{\mathrm{id}} \mathbb{G}_a^{\mathrm{perf}}$.
Then by \Cref{transmutation of id}, we have a natural equivalence of functors 
$\mathrm{Tm}(\id) \simeq (-)^{\flat} \otimes_k B \colon \mathrm{QRSP}_k \to \mathrm{Alg}_B$.
Let us denote this functor by $\mathfrak{G}$.
\end{example}{}

Since $\mathbb{G}_a^{\mathrm{perf}} \xrightarrow{\mathrm{id}} \mathbb{G}_a^{\mathrm{perf}}$ is
the initial object in $(\mathbb{G}_a^{\mathrm{perf}}-\mathrm{Mod}_*)^{\op}$.
By the above example, it follows that \Cref{defoft} naturally lifts to a functor
\begin{equation}\label{equ}
  \mathrm{Tm}\colon \Big(\mathbb{G}_a^{\mathrm{perf}}-\mathrm{Mod}_*\Big)^{\op} \to \mathrm{Fun}(\mathrm{QRSP}_k, \mathrm{Alg}_B)_{\mathfrak{G}/}, \end{equation}
where the latter denotes the under category associated with $\mathfrak{G}$.

\section{Nilpotent quasi-ideals in $\mathbb{G}_a^\perf$}
In this section, we restrict attention to a special class of augmeted $\mathbb{G}_a^\perf$-modules, which we call nilpotent quasi-ideals in $\mathbb{G}_a^\perf$, for which the transmutation construction will be particularly well-behaved. The notion of a quasi-ideal is due to Drinfeld \cite{drinfeld}.
\begin{definition}
\label{otime}
We define a full subcategory $\mathrm{Fun}(\mathrm{QRSP}_k, \mathrm{Alg}_B)^{\otimes}_{\mathfrak{G}/}$ of $\mathrm{Fun}(\mathrm{QRSP}_k, \mathrm{Alg}_B)_{\mathfrak{G}/}$ spanned by functors $F$ that satisfies the three conditions below.

\begin{enumerate}
    \item The map $\mathfrak{G}(S) \to F(S)$ is an isomorphism for every perfect ring $S$.

    \item The natural map $F\left(\frac{k[x^{1/p^\infty}]}{x} \right) \otimes_{F(k)} F(S) \xrightarrow{} F \left(\frac{S[x^{1/p^\infty}]}{x} \right)$ is an isomorphism for every perfect ring $S$.

    \item The natural map
    $F\left (\frac{S[x^{1/p^\infty}]}{x} \right) \otimes_{F(S)} F\left (\frac{S[x^{1/p^\infty}]}{x} \right) \xrightarrow{}F\left (\frac{S[x^{1/p^\infty}]}{x} \otimes_S \frac{S[x^{1/p^\infty}]}{x} \right)$ is an isomorphism for every perfect ring $S$.
\end{enumerate}{}
\end{definition}{}

It is worth pointing out that the functor $\mathfrak{G}$ itself does not satisfy either of
the conditions (2) and (3) above: Indeed we have 
$\mathfrak{G}(S[x^{1/p^\infty}]/(x)) = S[\![x^{1/p^\infty}]\!] \otimes_k B$ for any perfect $k$-algebra $S$,
so only a ``topological'' version of (2) and (3) holds true.
Below, we shall single out a certain full subcategory of $\mathbb{G}_a^{\mathrm{perf}}-\mathrm{Mod}_*$,
whose objects $d \colon M \to \mathbb{G}_a^{\mathrm{perf}}$ have transmutations $\mathrm{Tm}(d)$ satisfying the three conditions above. This leads us to the notion of \emph{nilpotent} quasi-ideals (see \cite[Def.~3.4.11]{M22}), where a certain nilpotency condition is used to discretise the ``topological" issue noted earlier.

\begin{definition}[Nilpotent quasi-ideals] 
A quasi-ideal in $\mathbb{G}_a^{\mathrm{perf}}$ is an augmented $\mathbb{G}_a^{\perf}$-module
$d: M \to \mathbb{G}_a^{\mathrm{perf}}$,
such that $d(y) x = d(x) y$ for all (scheme theoretic) points $x, y$ of $M$.

A nilpotent quasi-ideal in $\mathbb{G}_a^{\mathrm{perf}}$ is a quasi-ideal
$d\colon M \to \mathbb{G}_a^{\mathrm{perf}}$ such that the image of $x$ under the map 
$d^{*} \colon B[x^{1/p^\infty}] \to \Gamma (M, \mathcal{O})$ induced by $d$ is nilpotent.
We denote the corresponding full subcategory by 
$\mathcal{N}\mathrm{QID}-\mathbb{G}_a^{\mathrm{perf}} \subset \mathbb{G}_a^{\mathrm{perf}}-\mathrm{Mod}_*$.
\end{definition}

\begin{remark}
By using \Cref{transmutation on generator},
\Cref{why quasi-ideal} and diagrams before and inside \cite[Definition 3.2.10]{M22},
we see that if $\mathrm{Tm}(d)$ satisfies the conditions in \Cref{otime},
then the augmented module $d \colon M \to \mathbb{G}_a^{\perf}$ is necessarily
a quasi-ideal.
\end{remark}{}

\begin{remark}
\label{check nilpotence remark}
In the above definition, since $d$ is a map of $\mathbb{G}_a^{\mathrm{perf}}$-modules, the map 
$d^*\colon B[x^{1/p^\infty}] \to \Gamma (M, \mathcal{O})$ is a map of $\mathbb{N}[1/p]$-graded rings.
Note that since $1$ has degree zero and $d^*(x)$ has degree $1$, it follows that $d^*(x)$ 
is nilpotent if and only if $1 - d^*(x)$ is a unit.  
\end{remark}

\begin{proposition}
\label{Tm of nilp satisfies axioms}
The functor from \cref{equ} factors to give a functor 
$$\mathrm{Tm}\colon  \left( \mathcal{N}\mathrm{QID}-\mathbb{G}_a^{\mathrm{perf}} \right)^{\mathrm{op}}
\to \mathrm{Fun}(\mathrm{QRSP}_k, \mathrm{Alg}_B)^{\otimes}_{\mathfrak{G}/}. $$
\end{proposition}
\begin{proof}
This amounts to checking that the three conditions in \cref{otime} are satisfied. The first one follows from \cref{useexample}, and the latter two follows from the proof of \cite[Prop.~3.4.16, 3.4.17]{M22}. 
%\textcolor{red}{Let's sketch the idea: 
%this identification can be used to prove condition (2) of the result of transmutation,
%whereas a $2$-variable version can be used to prove condition (3).}
\end{proof}

\begin{remark}
\label{decompletion for NQID}
Let us sketch the idea of showing that transmutation of a nilpotent quasi-ideal
satisfies the second condition of \Cref{otime}:
Recall that the transmutation is defined via a moduli interpretation (see \Cref{sections}),
so we need to show that the two corresponding moduli functors are naturally isomorphic.
The only subtlety is the domains of these moduli functors, 
assume that $d^*(x^N) = 0$, then we just need to observe the following identification of categories
\[
\{S[\![x^{1/p^\infty}]\!]\text{-algebras} \mid x^N = 0\} \cong \{S[x^{1/p^\infty}]\text{-algebras} \mid x^N = 0\},
\]
which gives the identification of domains of these moduli functors.
Using the same logic, one checks that for $(M, d) \in \mathbb{G}_a^{\mathrm{perf}}-\mathrm{Mod}_*$
such that $d^*(x^N) = 0$ for some $N \in \mathbb{N}$. 
Then for $(S, I) = (k[\![x^{1/p^\infty}]\!], (x)) \in \mathscr{P}I$, the analog of \Cref{transmutation on generator} 
still holds true, namely $\mathrm{Sec}_{(S, I)}(d) \cong M$ as schemes over $\mathbb{G}_a^{\mathrm{perf}}$.
\end{remark}

\begin{construction}
Consider the quasi-ideal in $\mathbb{G}_{a, k}^{\mathrm{perf}}$
$$\alpha^\flat \coloneqq \mathrm{Ker}(\mathbb{G}_a^{\mathrm{perf}} \to \mathbb{G}_a).$$ 
Now let $F \in \mathrm{Fun}(\mathrm{QRSP}_k, \mathrm{Alg}_B)^{\otimes}_{\mathfrak{G}/}$, we may apply
$\Spec(F(-))$ to the map $k[x^{1/p^\infty}] \twoheadrightarrow k[x^{1/p^\infty}]/(x) \cong \mathcal{O}(\alpha^\flat)$.
By the conditions in \cref{otime}, the resulting object is a quasi-ideal in $\mathbb{G}_a^{\mathrm{perf}}$.
Next we claim that the resulting quasi-ideal is nilpotent: Using \Cref{check nilpotence remark}
it is equivalent to checking the image of $1 - x$ being a unit in $F(k[x^{1/p^\infty}]/(x))$.
Note that the map $k[x^{1/p^\infty}] \twoheadrightarrow k[x^{1/p^\infty}]/(x)$
factors through the $x$-adic completion of the source, which is perfect and in which $1-x$ is already a unit,
so our claim follows from the first condition in \cref{otime}.
Since $\Spec(-)$ is contravariant, the above ``restriction to $\alpha^{\flat}$'' process defines a functor
$$\mathrm{Res}\colon \mathrm{Fun}(\mathrm{QRSP}_k, \mathrm{Alg}_B)^{\otimes}_{\mathfrak{G}/} \to \left( \mathcal{N}\mathrm{QID}-\mathbb{G}_a^{\mathrm{perf}} \right)^{\mathrm{op}}.$$ 
\end{construction}

\begin{proposition}[Full faithfulness of transmutation]
\label{transfullyfaithful}
Let $B$ be a $k$-algebra fixed as before. The functors 
$$\mathrm{Tm}\colon \left(\mathcal{N}\mathrm{QID}-\mathbb{G}_a^{\mathrm{perf}} \right)^{\mathrm{op}}
\substack{\xrightarrow{} \\ \xleftarrow{}} \mathrm{Fun}(\mathrm{QRSP}_k, \mathrm{Alg}_B)^{\otimes}_{\mathfrak{G}/}\colon \mathrm{Res}$$ are adjoint to each other and
the left adjoint $\mathrm{Tm}$ is fully faithful.
\end{proposition}

\begin{proof}
Using \Cref{Tm of nilp satisfies axioms} and \Cref{decompletion for NQID}, we first 
get an equivalence $\id \xrightarrow{\sim} \mathrm{Res} \circ \mathrm{Tm}$, which is the desired unit.
Next we need to construct a counit and check compatibility conditions.
The key idea is to use the universal property of transmutation construction to produce a natural transformation $\mathrm{Tm} (\mathrm{Res}\, F) \to F$ for any $F \in \mathrm{Fun}(\mathrm{QRSP}_k, \mathrm{Alg}_B)^{\otimes}_{\mathfrak{G}/}$. 
For details, see the proof of \cite[Prop.~3.2.18, Prop.~3.4.20]{M22}. 
Note that the unit map $\id \xrightarrow{\sim} \mathrm{Res} \circ \mathrm{Tm}$ is an equivalence,
we see that the left adjoint $\mathrm{Tm}$ is fully faithful.
 \end{proof}

\begin{construction}[Pullback functors] For any ring scheme $R$, one can define a notion of augmented $R$-modules as in \cref{augmodule}. For a map of ring schemes $u\colon R_1 \to R_2$, and an augmented $R_2$-module $d\colon M \to R_2$, the scheme theoretic pullback $M\times _{R_2} R_1 \to R_1$ is naturally an augmented $R_1$-module, which we denote by $u^* d.$
    
\end{construction}

\begin{proposition}
Let $u\colon \mathbb{G}_a^{\mathrm{perf}} \to \mathbb{G}_a$. Then the pullback functor from augmented $\mathbb{G}_a$-modules to augmented $\mathbb{G}_a^{\mathrm{perf}}$-modules is fully faithful.  
\end{proposition}

\begin{proof}
See \cite[Prop.~2.2.17]{M22}.   
\end{proof}{}

Let us give some examples of nilpotent quasi-ideals and their corresponding
transmutations.

\begin{example}
\label{Tm inverse of O}
The functor $\mathrm{id \otimes B}\colon \mathrm{QRSP}_k \to \mathrm{Alg}_B$ that sends $S \mapsto S \otimes_k B$ is isomorphic to $\mathrm{Tm}(\alpha^\flat \to \mathbb{G}_a^{\mathrm{perf}}).$
\end{example}{}

\begin{example}[de Rham cohomology, see {\cite[Prop.~4.0.3]{M22}}]
\label{exa1}
Consider the quasi-ideal in $\mathbb{G}_a$ given by $W[F]\to \mathbb{G}_a$, and let us pullback this quasi-ideal
along the map $u\colon \mathbb{G}_a^{\mathrm{perf}} \to \mathbb{G}_a$ of ring schemes.
Then the derived de Rham cohomology, viewed as an object $\mathrm{dR} \in \mathrm{Fun}(\mathrm{QRSP}_k, \mathrm{Alg}_B)^{\otimes}_{\mathfrak{G}/}$ is equivalent to $\mathrm{Tm}\left(u^*(W[F]\to \mathbb{G}_a)\right)$.
\end{example}

\begin{construction}[Frobenius pushforwards]
\label{frobpush}
Let $\varphi\colon \mathbb{G}_a^{\mathrm{perf}} \to \mathbb{G}_a^{\mathrm{perf}}$ 
be the relative Frobenius map over $B$. 
Given a $\mathbb{G}_a^{\mathrm{perf}}$-module $M$, 
one can consider the tensor product $M^{(1/p)} \coloneqq M \otimes_{\mathbb{G}_a^{\mathrm{perf}}, \varphi} \mathbb{G}_a^{\mathrm{perf}}$, 
which is representable by an affine scheme and is naturally a $\mathbb{G}_a^{\mathrm{perf}}$-module. 
Now given an augmented $\mathbb{G}_a^{\mathrm{perf}}$-module $d\colon M \to \mathbb{G}_a^{\mathrm{perf}}$, 
we may compose the maps 
$$  
M^{(1/p)} \to { \mathbb{G}_a^{\mathrm{perf}}}^{(1/p)}\xrightarrow{\varphi} \mathbb{G}_a^{\mathrm{perf}}
$$ 
to obtain another augmented $\mathbb{G}_a^{\mathrm{perf}}$-module that we denote by $\varphi_*d\colon  M^{(1/p)} \to \mathbb{G}_a^{\mathrm{perf}}$. 

On the other hand, given an object $\mathfrak{G} \to F$ of $\mathrm{Fun}(\mathrm{QRSP}_k, \mathrm{Alg}_B)_{\mathfrak{G}/}$, we set $$F^{(1/p)}(S) \coloneqq F(\varphi_{k,*} S).$$ There is a natural transformation of functors $\mathfrak{G}^{(1/p)} \to F^{(1/p)}.$ For any $S \in \mathrm{QRSP}_k$, the $k$-linear Frobenius map $S^\flat \to \varphi_{k,*} S^\flat$ induces a map $S^\flat \otimes_k B \to \varphi_{k,*} S^\flat \otimes_k B$ that gives a natural transformation $\mathfrak{G} \xrightarrow{\mathrm{Frob}}\mathfrak{G}^{(1/p)}$. The composition $\mathfrak{G} \xrightarrow{\mathrm{Frob}} \mathfrak{G}^{(1/p)} \to F^{(1/p)}$ is naturally an object of $\mathrm{Fun}(\mathrm{QRSP}_k, \mathrm{Alg}_B)_{\mathfrak{G}/}$, which will be denoted as $\mathrm{Frob}_*(\mathfrak{G} \to F)$. By construction, transmutation is compatible with the two operations we discussed. Namely, we have 
$$\mathrm{Tm} (\varphi_*d) \simeq \mathrm{Frob}_*(\mathrm{Tm} (d)).$$

By the same logic of \Cref{Frob for augmented Gaperf mod}, using the equivalence
$\id \xrightarrow{\sim} \mathrm{Res} \circ \mathrm{Tm}$ from \Cref{transfullyfaithful}, for any 
$(M, d) \in \mathcal{N}\mathrm{QID}-\mathbb{G}_a^{\mathrm{perf}}$, we get a map
$\varphi_M \colon \varphi_* d \to d$ in $\mathcal{N}\mathrm{QID}-\mathbb{G}_a^{\mathrm{perf}}$,
which corresponds to the map $\mathrm{Tm} (d) \to \mathrm{Frob}_*(\mathrm{Tm} (d))$ given by applying functoriality to
the relative Frobenius of the input $S \to \varphi_* S$.
\end{construction}{}

In general the map $\varphi_M$ from above is not transparent, but in the case of
$\alpha^{\flat}$ it is easy to determine.

\begin{lemma}
\label{Frob on alpha flat}
The map $\varphi_{\alpha^{\flat}}$ is given by the ``$B$-linear Frobenius'':
\[
\alpha^{\flat, (1/p)} \cong \Spec(B[x^{1/p^\infty}]/(x)) \xrightarrow{x^{1/p^{i+1}} \mapsto x^{1/p^i}} 
\Spec(B[x^{1/p^\infty}]/(x)) \cong\alpha^{\flat} .
\]
\end{lemma}

Note that the element $x^i \in \mathcal{O}(\alpha^{\flat, (1/p)})$ has degree $i/p$,
therefore the relative Frobenius is indeed a graded map.

\begin{proof}
This follows from the commutative diagram from \Cref{Frob for augmented Gaperf mod}.
\end{proof}

\begin{remark}
\label{remarkcomp}
Note that the category of augmented $\mathbb{G}_a^{\mathrm{perf}}$-modules have pullbacks. In fact, the forgetfull functor ${\mathbb{G}_a^{\mathrm{perf}}}-\mathrm{Mod}_* \to {\mathbb{G}_a^{\mathrm{perf}}}-\mathrm{Mod}$ that sends $(M \to \mathbb{G}_a^{\mathrm{perf}}) \mapsto M$ preserves pullbacks. Furthermore, the category $\mathcal{N}\mathrm{QID}-\mathbb{G}_a^{\mathrm{perf}}$ is closed under pullbacks and the (contravariant) transmutation functor
$$\mathrm{Tm}\colon  \left( \mathcal{N}\mathrm{QID}-\mathbb{G}_a^{\mathrm{perf}} \right)^{\mathrm{op}} \to \mathrm{Fun}(\mathrm{QRSP}_k, \mathrm{Alg}_B)^{\otimes}_{\mathfrak{G}/} $$ takes pullbacks to pushouts. 
\end{remark}{}

\begin{example}
The global sections of the underlying $\mathbb{G}_a^{\mathrm{perf}}$-module $u^* W[F]$ gives a $\mathbb{N}[1/p]$-graded 
$B$-Hopf algebra, whose underlying $\mathbb{N}[1/p]$-graded 
$B$-algebra is explicitly described as  
$$\frac{B[x_0^{1/p^\infty}, x_1, \ldots, x_i, \ldots]}{x_i^p},$$
%\textcolor{red}{missing comultiplication formulas}
where $\deg x_i = p^i$ for all $i \in \mathbb{N}$.  The global sections of $\alpha^{\flat}$ can be explicitly described as $B[x_0^{1/p^\infty}]/x_0$, where $\deg x_0 = 1$. Note that there is no map 
$u^* W[F] \dashrightarrow \alpha^\flat$ of augmented $\mathbb{G}_a^{\mathrm{perf}}$-modules.
\end{example}

\begin{remark}
\label{rmk2}
%Suppose for simplicity again that $B= k$.
The $0$-th conjugate filtration (\Cref{conjugate filtration}) gives rise to a map
$\id \otimes_k B \to F_* \dR \otimes_k B$
in $\mathrm{Fun}(\mathrm{QRSP}_k, \mathrm{Alg}_B)^{\otimes}_{\mathfrak{G}/}$. By \Cref{Tm inverse of O}, the source is $\mathrm{Tm}(\alpha^{\flat})$.
By \Cref{exa1} and \Cref{frobpush}, the target is $\mathrm{Tm}(u^* W[F])^{(1/p)})$.
Then according to \cref{transfullyfaithful}, the $0$-th conjugate filtration
must be given, via transmutation, by a map $(u^* W[F])^{(1/p)} \to \alpha^\flat$
of nilpotent quasi-ideals in $\mathbb{G}_a^{\perf}$.

There is unique such map because the augmentation $d \colon \alpha^\flat \hookrightarrow \mathbb{G}_a^{\perf}$
is a closed immersion. Below let us explicate this map.
Note that the $\mathbb{N}[1/p]$-graded $B$-algebra underlying $(u^* W[F])^{(1/p)}$ can be described as 
$$\frac{B[x_0^{1/p^\infty}, x_1, \ldots, x_i, \ldots]}{x_i^p},$$ where $\deg x_i = p^{i-1}$ for all $i \in \mathbb{N}$.
Then the map must be given by
$$
\mathcal{O}(\alpha^{\flat}) \cong \frac{B[x_0^{1/p^\infty}]}{x_0} \to 
\frac{B[x_0^{1/p^\infty}, x_1, \ldots, x_i, \ldots]}{x_i^p},
\text{ } x_0^i \mapsto x_0^{pi},
$$ 
as it has to be a map of schemes over $\mathbb{G}_a^{\perf}$.
%that sends $x_0^i \mapsto x_0^{pi}$ defines a map of augmented 
%$\mathbb{G}_a^{\mathrm{perf}}$-modules $(u^* W[F])^{(1/p)} \to \alpha^\flat$. 
%By transmutation, this gives a natural (semilinear) transformation $S \to \mathrm{dR}(S)$ for every 
%$S \in \mathrm{QRSP}_{\mathbb{F}_p}$ which corresponds to the $0$-th conjugate filtration on de Rham cohomology $F_*\mathrm{dR}$:
%This follows from the more general observation that there is a unique natural transformation $\id \to F_* \dR$ in $\mathrm{Fun}(\mathrm{QRSP}_k, \Alg_k)_{\mathfrak{G}/}^\otimes$ by \cref{transfullyfaithful}, as there is a unique map of $\mathbb{G}_a^\mathrm{perf}$-modules $(u^* W[F])^{1/p} \to \alpha^\flat$.
\end{remark}{}

Now we describe an augmented $\mathbb{G}_a^{\mathrm{perf}}$-module that will be used in the next proposition.

\begin{construction}
 Let $W[F]^{(p)}$ denote the $\mathbb{G}_a$-module obtained from the $\mathbb{G}_a$-module $W[F]$ via restriction of scalars along (the $B$-linear) Frobenius map $\mathbb{G}_a \to \mathbb{G}_a$. Consider the augmented $\mathbb{G}_a$-module $W[F]^{(p)} \xrightarrow{0} \mathbb{G}_a.$ Pulling back along $u\colon \mathbb{G}_a^{\mathrm{perf}} \to \mathbb{G}_a$ and applying the Frobenius pushforward \cref{frobpush}, we obtain an augmented $\mathbb{G}_a^{\mathrm{perf}}$-module $\varphi_*u^* (W[F]^{(p)} \xrightarrow{0} \mathbb{G}_a)$. By construction, it is a nilpotent quasi-ideal in $\mathbb{G}_a^{\mathrm{perf}}.$
\end{construction}

For later purpose, let us explicate the outcome of the above construction:

\begin{remark}
\label{Computing the complicated Gaperf module}
Similar to \Cref{interprete Gaperf module},  
a $\mathbb{G}_a$-module is the same as an $\mathbb{N}$-graded bi-commutative Hopf $B$-algebra satisfying certain extra compatibility (see \cite[Remark 2.2.8]{M22}).
Under this identification, the $\mathbb{G}_a$-module $W[F]^{(p)}$ over $B$ is given by $B[x_0, x_1, \ldots]/(x_i^p)$, with grading $\deg x_i = p^{i+1}$.
Its comultiplication is determined using the Witt vector summation, for instance the comultiplication sends
\[
x_0 \mapsto x_0 \otimes1 + 1 \otimes x_0, \text{  and } x_1 \mapsto x_1 \otimes 1 + 1 \otimes x_1 - 
\sum_{i = 1}^{p-1} \frac{\binom{p}{i}}{p} x_0^{i} \otimes x_0^{p-i}, \ldots.
\]
The augmentation of $W[F]^{(p)} \xrightarrow{0} \mathbb{G}_a$ is simply given by $d: B[t] \to B[x_0, x_1, \ldots]/(x_i^p)$ sending $t \mapsto 0$.

Next we consider the effect of the functor $u^*$ to augmented $\mathbb{G}_a$-module $W[F]^{(p)} \xrightarrow{0} \mathbb{G}_a$.
The underlying $\mathbb{G}_a^{\mathrm{perf}}$-module of $u^* (W[F]^{(p)})$ is given by push out along $u: B[t] \to B[x^{1/p^\infty}]$ that sends $t \mapsto x$ 
\[
B[x_0, x_1, \ldots]/(x_i^p) \otimes_{d, B[t], u} B[x^{1/p^\infty}] = B[x^{1/p^\infty}, x_0, x_1, \ldots]/(x, x_i^p),
\]
with grading $\deg x^{1/p^i} = 1/p^i$ and $\deg x_i = p^{i+1}$. The comultiplication sends $x \mapsto x \otimes 1 + 1 \otimes x$, and respects
the above comultiplication formula on the $x_i$'s.

Lastly after we apply the functor $\varphi_*$, the underlying Hopf $B$-algebra of $\varphi_*u^* (W[F]^{(p)} \xrightarrow{0} \mathbb{G}_a)$ is the same as above, but the grading is divided by $p$,
so we have $\deg x^{1/p^i} = 1/{p^{i+1}}$ and $\deg x_i = p^i$. The augmentation is given by
\[
B[x^{1/p^\infty}] \to B[x^{1/p^\infty}, x_0, x_1, \ldots]/(x, x_i^p), \text{ where } x^{1/p^i} \mapsto x^{1/p^{i-1}}.
\]
\end{remark}

\section{Automorphisms of $F^*F_*\mathrm{dR}$ and decomposition}

Now we can state the main observation involving transumtation that is relevant for the results of our paper. 

\begin{proposition}
\label{transm1}
 Let $k$ be a perfect field and $B$ be a $k$-algebra. Then we have an equivalence in $\mathrm{Fun}(\mathrm{QRSP}_k, \mathrm{Alg}_B)^{\otimes}_{\mathfrak{G}/}$
 $$F^* F_* \mathrm{\mathrm{dR}} \otimes_k B \simeq \mathrm{Tm}(\varphi_*u^* (W[F]^{(p)} \xrightarrow{0} \mathbb{G}_a)).$$  
\end{proposition}{}

\begin{proof}
By \Cref{base change remark}, we can assume $B \coloneqq k$. Let $d\colon u^*(W[F] \to \mathbb{G}_a)$; 
by \cref{exa1}, it follows that $\mathrm{Tm}(d) \simeq \mathrm{dR}$.
Let $S\in \mathrm{QRSP}_k$. 
By \cref{rmk2}, the $k$-linear map $S \to \varphi_{k,*} \mathrm{dR}(S)$ is induced by applying transmutation to the map
$(u^* W[F])^{(1/p)} \to \alpha^\flat$ of augmented $\mathbb{G}_a^{\mathrm{perf}}$-modules. 
By \Cref{Tm inverse of O} and \Cref{Frob on alpha flat},
the $k$-linear Frobenius map $S \to \varphi_{k,*}S$ is induced by applying transmutation to the 
``$B$-linear Frobenius'' $(\alpha^{\flat})^{(1/p)} \to \alpha^\flat$. 
By \cref{remarkcomp}, it follows that 
$$F^* F_* \mathrm{dR} \simeq \mathrm{Tm}((u^* W[F])^{(1/p)} \times_{\alpha^\flat} (\alpha^\flat)^{(1/p)}). $$
The proposition now follows from observing that 
\begin{equation}\label{impobservation}
 (u^* W[F])^{(1/p)} \times_{\alpha^\flat} (\alpha^\flat)^{(1/p)} \simeq \varphi_*u^* (W[F]^{(p)} \xrightarrow{0} \mathbb{G}_a),   
\end{equation}
which we explain below.

Let $P$ denote the augmented $\mathbb{G}_a^\perf$-module defined as the pullback
$$P\coloneqq u^*W[F] \times_{u^*\alpha_p} \alpha^{\flat}.$$ 
Since $(u^*\alpha_p)^{(1/p)} \simeq \alpha^\flat$, it follows that the left hand side in \cref{impobservation} 
is naturally isomorphic to $\varphi_* P.$ By commuting products and pullbacks, we also have 
$$P \simeq u^* (W[F]\times_{\alpha_p} (0)),$$ where $(0)$ denotes the quasi-ideal in $\mathbb{G}_a$ given by $ 0 \to \mathbb{G}_a.$ Now the Verschiebung operator $V$ induces an exact sequence
of $\mathbb{G}_a$-modules
$$0 \to W[F]^{(p)}\xrightarrow{V} W[F] \to \alpha_p \to 0,$$ which promotes to an isomorphism 
$$W[F]\times_{\alpha_p} (0) \simeq (W[F]^{(p)} \xrightarrow{0} \mathbb{G}_a) $$ of augmented $\mathbb{G}_a$-modules. This finishes the proof.
\end{proof}

\begin{proposition}
\label{endoauto}
Let $k$ be a perfect field and $B$ be a $k$-algebra. Then $\varphi_*u^* (W[F]^{(p)} \xrightarrow{0} \mathbb{G}_a)$ 
has endomorphism monoid given by $(B, \cdot)$. In particular, its automorphism group is $B^\times$. 
Moreover, denote the underlying $\mathbb{G}_a^{\perf}$-module by $M$,
the action of $b \in B$ sends $x_0 \in \mathcal{O}(W[F]^{(p)}) \subset 
\mathcal{O}(M)$ to $b \cdot x_0$.
\end{proposition}

\begin{proof}
Endomorphisms of $\varphi_*u^* (W[F]^{(p)} \xrightarrow{0} \mathbb{G}_a)$ are equivalent to
endomorphisms of $u^* (W[F]^{(p)} \xrightarrow{0} \mathbb{G}_a)$. 
Further, by full faithfulness of $u^*$, it is equivalent to endomorphisms of 
$(W[F]^{(p)} \xrightarrow{0} \mathbb{G}_a).$ 
This is equivalent to endomorphisms of $W[F]^{(p)}$ as a $\mathbb{G}_a$-module, which is further equivalent to endomorphisms of 
$W[F]$ as a graded group scheme. By graded Cartier duality, 
see \cite[Ex.~2.4.9 and Prop.~2.4.10]{M22}, 
the latter is equivalent to endomorphisms of $\mathbb{G}_a$ as a graded group scheme
(with coordinate given by the functional $(x_0)^{\vee}$).
Our proposition follows from the observation that endomorphism of the graded group scheme is
given by the monoid $(B,\cdot)$, where $b \in B$ sends $(x_0)^{\vee} \mapsto b \cdot (x_0)^{\vee}$.
\end{proof}{}

\begin{proof}[Proof of \cref{mainthm1}] 
By \cref{qsyndescent}, it is equivalent to considering endomorphisms {automorphisms}
of the left Kan extension
$F^*F_*\mathrm{dR} \otimes_k B \in \mathrm{Fun}(\mathrm{QRSP}_k, \mathrm{Alg}_B)$ (by restricting the target category from quasisyntomic rings to $\mathrm{QRSP}_k$). 
By \Cref{automorphism is linear over Fil0}, it is further equivalent to
considering automorphisms of
$F^*F_*\mathrm{dR} \otimes_k B \in \mathrm{Fun}(\mathrm{QRSP}_k, \mathrm{Alg}_B)^{\otimes}_{\mathfrak{G}/}$.
By \Cref{transm1}, 
the functor $F^*F_*\mathrm{dR} \otimes_k B$ is the transmutation of the nilpotent quasi-ideal $\varphi_*u^* (W[F]^{(p)} \xrightarrow{0} \mathbb{G}_a)$.
By \cref{transfullyfaithful}, we are reduced to computing endomorphisms {automorphisms} of 
$\varphi_*u^* (W[F]^{(p)} \xrightarrow{0} \mathbb{G}_a)$,
which is the content of \cref{endoauto}.
\end{proof}

\begin{remark}
Following the above reasoning, it is possible to determine the whole endomorphism monoid of 
$F^* F_* \mathrm{dR} \otimes_k B$, let us indicate the necessary steps and leave the fun of 
computing all endomorphisms to the interested readers.
First of all, from \Cref{automorphism is linear over Fil0}, there is a natural map
\[
\mathrm{End}(F^* F_* \mathrm{dR} \otimes_k B) \rightarrow \mathrm{End}(\mathcal{O} \otimes_k B)
\]
given by restricting to the $0$-th conjugate filtration.
Moreover the proof there can be generalized to help computing the later monoid, it should be
a submonoid of $\mathrm{Frob}^\mathbb{N}$ depending on the cardinality of $k$.
The above map admits a section, given by applying various powers of Frobenius to the input algebra.
Lastly one needs to compute the fiber above $\mathrm{Frob}^i$, this is the same as computing
the following homomorphism set
\[
\mathrm{Hom}((\mathrm{Frob}^i)^*(F^* F_* \mathrm{dR} \otimes_k B), F^* F_* \mathrm{dR} \otimes_k B)
\]
in the category $\mathrm{Fun}(\mathrm{QRSP}_k, \mathrm{Alg}_B)^{\otimes}_{\mathfrak{G}/}$.
Finally, similar to \Cref{transm1}, one can get a nilpotent quasi-ideal corresponding to the
source and then use the proof of \Cref{endoauto} to compute the above homomorphism in the category of
nilpotent quasi-ideals.
\end{remark}

By \Cref{mainthm1}, similar to the discussion before \cite[Theorem 5.4]{LM24}, 
we obtain a functorial action of $\mathbb{G}_m$ on $F^* F_* \mathrm{dR} 
\colon(\mathrm{Sch}_k^{\mathrm{sm}})^{\mathrm{op}} \to \mathrm{CAlg}(k)$.

\begin{theorem}[{c.f.~\cite[Theorem 5.4]{LM24}}]\label{thm5.4}
The functorial $\mathbb{G}_m$-action preserves the conjugate filtration defined in \Cref{conjugate filtration},
and $\gr^{\mathrm{conj}}_i F^* F_* \dR [i]$ is a $\mathbb{G}_m$-representation purely of weight $i$.
\end{theorem}

\begin{proof}
The preservation of conjugate filtration and the triviality of the action on $\Fil^{\conj}_0$ 
were proved in \Cref{automorphism is linear over Fil0}.
Following the same proof strategy of \cite[Theorem 5.4]{LM24}, it would suffice to explicate
the $\mathbb{G}_m$-action on $\mathrm{R\Gamma}(\mathbb{A}^1_k, F^* F_* \mathrm{dR})$
and exhibit a nonzero weight $1$ element in $\mathrm{H}^1(\mathbb{A}^1_k, F^* F_* \mathrm{dR})$.

Recall that the $\mathbb{G}_m$-action on $F^* F_* \mathrm{dR} \otimes_k B(\text{smooth algebras})$
comes, via descent \Cref{qsyndescent}, from the $\mathbb{G}_m$-action on $F^* F_* \mathrm{dR} \otimes_k B(\text{qrsp algebras})$.
The action on $F^* F_* \mathrm{dR} \otimes_k B(\text{qrsp algebras})$ then, due to \Cref{transm1}, comes from \Cref{defoft}.
Therefore we arrive at the following explicit cosimplicial presentation in a $\mathbb{G}_m$-equivariant fashion:
\[
F^* F_* \mathrm{dR}(k[x]) = \lim_{[n] \in \Delta} F^* F_* \mathrm{dR}(k[x^{1/p^{\infty}}]^{\otimes_{k[x]} (n + 1)})
= \lim_{[n] \in \Delta} \mathrm{Tm}(\varphi_*u^* (W[F]^{(p)} \xrightarrow{0} \mathbb{G}_a))
(k[x_0^{1/p^{\infty}}, \ldots, x_{n}^{1/p^{\infty}}]/(x_i - x_j)).
\]
%Here we abuse notation by viewing the category of semiperfect algebras as a full subcategory of
%$\mathscr{P}I$ from \Cref{sections}.
For simplicity, let us denote $G(-) \coloneqq \mathrm{Tm}(\varphi_*u^* (W[F]^{(p)} \xrightarrow{0} \mathbb{G}_a))(-)$.
We are interested in degree $1$ cohomology, to this end, let us explicate the first three terms in a convenient presentation
together with the coface maps: They are given by
\[
\xymatrix{
G\Big(k[x^{1/p^{\infty}}]\Big) \ar@<.4ex>[r] \ar@<-.4ex>[r] &
G\Big(k[x^{1/p^{\infty}}] \otimes_k k[y^{1/p^{\infty}}]/(y)\Big) \ar@<0.8ex>[r] \ar[r] \ar@<-.8ex>[r] &
G\Big(k[x^{1/p^{\infty}}] \otimes_k k[y^{1/p^{\infty}}]/(y) \otimes_k k[z^{1/p^{\infty}}]/(z)\Big),
}
\]
the first two arrows are induced by $x \mapsto x$ and $x \mapsto x+y$ respectively; whereas the latter three arrows are induced by
$(x, y) \mapsto (x, y)$, $(x, y) \mapsto (x, y+z)$, and $(x, y) \mapsto (x + y, z)$ respectively.

Combining \Cref{transmutation on generator} and \Cref{transmutation of id},
we see that for any augmented $\mathbb{G}_a^{\perf}$-module $M \xrightarrow{d} \mathbb{G}_a^{\perf}$, applying
$\mathrm{Tm}(d)(-)$ to $k[x^{1/p^{\infty}}] \to k[x^{1/p^{\infty}}]/(x)$ gives
$d^* \colon k[x^{1/p^{\infty}}] \to \mathcal{O}(M)$.
According to \Cref{Tm of nilp satisfies axioms}, the functor $G(-)$ satisfies the axioms listed in \Cref{otime}.
Therefore, using the computation made in \Cref{Computing the complicated Gaperf module},
the above presentation becomes:
\[
\xymatrix{
k[x^{1/p^{\infty}}] \ar@<.4ex>[r] \ar@<-.4ex>[r] &
k[x^{1/p^{\infty}}] \otimes_k \frac{k[y^{1/p^\infty}, y_i; i \in \mathbb{N}]}{(y, {y_i}^p)} \ar@<0.8ex>[r] \ar[r] \ar@<-.8ex>[r] &
k[x^{1/p^{\infty}}] \otimes_k \frac{k[y^{1/p^\infty}, y_i; i \in \mathbb{N}]}{(y, {y_i}^p)} \otimes_k 
\frac{k[z^{1/p^\infty}, z_i; i \in \mathbb{N}]}{(z, {z_i}^p)},
}
\]
the first two arrows are given by $x \mapsto x$ and $x \mapsto x+y$ respectively; whereas the latter three arrows are
given by $\id \otimes 1$, $\id \otimes \text{comultiplication}$, and $(x \mapsto x + y) \otimes (y, y_i \mapsto z, z_i)$ respectively.
In \Cref{Computing the complicated Gaperf module} we have seen that 
the comultiplication sends $y_0 \mapsto (y_0 + z_0)$, hence we see that $y_0$ defines a nonzero cohomology class.
By the last part of \Cref{endoauto}, we see that $y_0$ has weight $1$.
Note that $\mathrm{H}^1(F^* F_* \mathrm{dR}(k[x]))$
is a free rank 1 module over $k[x]$, we see that it is indecomposable, hence the whole $\mathrm{H}^1$ has weight $1$.
The rest follows from K\"{u}nneth formula consideration for polynomial algebras, and the fact that $\Fil^{\conj}_{\bullet}(F^* F_* \mathrm{dR})$
is obtained by left Kan extension from polynomial algebras.
\end{proof}

Since any $\mathbb{G}_m$-representation decomposes canonically according to $\mathbb{G}_m$-weights, we get the following:
\begin{corollary}[c.f.~{\cite[Theorem 1.2]{Petrov25}}]
Let $X$ be a smooth $k$-scheme, then the Frobenius twisted de Rham complex is functorially formal:
\[
F^* F_* \mathrm{dR}_X \cong \bigoplus_{i \geq 0} F^* \Omega^i_{X^{(1)}/k}[-i].
\]
\end{corollary}

%\begin{remark}
%In \cite[Theorem 1.1]{Petrov25}, Petrov uses the above decomposition result to deduce
%Hodge-to-de Rham degeneration and Kodaira--Akizuki--Nakano vanishing for $F$-split varieties.
%Moreover Petrov actually prove the decomposition of $F_* \mathrm{dR}$ after base changing to $W(\mathcal{O})/p$,
%leading him to same conclusions even for quasi-$F$-split varieties!
%\end{remark}

\subsection*{Acknowledgments}
The authors are thankful to Bhargav Bhatt, Luc Illusie and Alexander Petrov for helpful comments and questions. The research was done when the second author was a postdoctoral fellow at the University of British Columbia.

%S.L.~is supported by the National Key R $\&$ D Program of China No.~2023YFA1009701 and
%the National Natural Science Foundation of China (No.~12288201). S.M. is supported by a postdoctoral fellowship at the University of British Columbia, Canada.

\bibliographystyle{amsalpha}
\bibliography{references}

\providecommand{\bysame}{\leavevmode\hbox to3em{\hrulefill}\thinspace}
\providecommand{\MR}{\relax\ifhmode\unskip\space\fi MR }
\providecommand{\MRhref}[2]{%
  \href{http://www.ams.org/mathscinet-getitem?mr=#1}{#2}
}
\providecommand{\href}[2]{#2}
\begin{thebibliography}{BMS19}

\bibitem[Bha12]{Bha12}
Bhargav Bhatt, \emph{p-adic derived de {Rham} cohomology}, 2012, arXiv:1204.6560.

\bibitem[Bha23]{fg}
\bysame, \emph{{Prismatic $F$-gauges}}, 2023, \url{https://www.math.ias.edu/~bhatt/teaching/mat549f22/lectures.pdf}.

\bibitem[BL22]{BL22a}
Bhargav Bhatt and Jacob Lurie, \emph{Absolute prismatic cohomology}, 2022, arXiv:2201.06120.

\bibitem[BL222]{BL22b}
\emph{The prismatization of $p$-adic formal schemes}, 2022, arXiv:2201.06124.

\bibitem[BMS19]{BMS19}
Bhargav Bhatt, Matthew Morrow, and Peter Scholze, \emph{Topological {H}ochschild homology and integral {$p$}-adic {H}odge theory}, Publ. Math. Inst. Hautes \'{E}tudes Sci. \textbf{129} (2019), 199--310. \MR{3949030}

\bibitem[DI87]{DI87}
Pierre Deligne and Luc Illusie, \emph{Rel\`evements modulo {$p^2$} et d\'{e}composition du complexe de de {R}ham}, Invent. Math. \textbf{89} (1987), no.~2, 247--270. \MR{894379}

\bibitem[Dri24]{drinfeld}
Vladimir Drinfeld, \emph{Prismatization}, Selecta Math. (N.S.) \textbf{30} (2024), no.~3, Paper No. 49, 150. \MR{4742919}

\bibitem[LM24]{LM24}
Shizhang Li and Shubhodip Mondal, \emph{On endomorphisms of the de {R}ham cohomology functor}, Geom. Topol. \textbf{28} (2024), no.~2, 759--802. \MR{4718127}

\bibitem[Mat16]{Mat16}
Akhil Mathew, \emph{The {G}alois group of a stable homotopy theory}, Adv. Math. \textbf{291} (2016), 403--541. \MR{3459022}

\bibitem[Mon22]{M22}
Shubhodip Mondal, \emph{{$\Bbb{G}^{\rm perf}_a$}-modules and de {R}ham cohomology}, Adv. Math. \textbf{409} (2022), Paper No. 108691, 72. \MR{4483247}

\bibitem[Pet25]{Petrov25}
Alexander Petrov, \emph{Decomposition of de {Rham} complex for quasi-{$F$}-split varieties}, 2025.

\end{thebibliography}
\end{document}